\documentclass[12pt,reqno]{amsart}

\usepackage{amssymb,amsmath}
\usepackage{amsfonts}
\usepackage{amsthm}
\usepackage{mathrsfs}
\usepackage[T1]{fontenc}% Brzd\k{e}k

\theoremstyle{plain} 
\newtheorem{theorem}{Theorem}

\theoremstyle{definition} 
\newtheorem{definition}[theorem]{Definition}
\newtheorem{remark}[theorem]{Remark}
\newtheorem{example}[theorem]{Example}
\newcommand{\R}{\ensuremath{\mathbb{R}}}

\newcommand{\T}{\mathcal{H}}
\newcommand{\N}{\ensuremath{\mathbb{N}}}
\newcommand{\C}{\ensuremath{\mathbb{C}}}

\numberwithin{equation}{section}
\numberwithin{theorem}{section}

\small\normalsize

\begin{document}

\title{Hyers-Ulam Stability For A Type Of Discrete Hill Equation}
\author[Anderson]{Douglas R. Anderson} 
\address{Department of Mathematics \\
         Concordia College \\
         Moorhead, MN 56562 USA\\ 0000-0002-3069-2816}
\email{andersod@cord.edu}
\author[Onitsuka]{Masakazu Onitsuka}
\address{Department of Applied Mathematics \\
Okayama University of Science \\
Okayama, 700-0005, Japan \\ 0000-0001-8598-0746}
\email{onitsuka@ous.ac.jp}


\begin{abstract} 
We establish the Hyers-Ulam stability of a second-order linear Hill-type $h$-difference equation with a periodic coefficient. Using results from first-order $h$-difference equations with periodic coefficient of arbitrary order, both homogeneous and non-homogeneous, we also establish a Hyers-Ulam stability constant. Several interesting examples are provided. As a powerful application, we use the main result to prove the Hyers-Ulam stability of a certain third-order $h$-difference equation with periodic coefficients of one form.
\end{abstract}

\maketitle\thispagestyle{empty}

%%%%%%%%%%%%%%%%% 
%               % 
% SECTION Intro % 
%               % 
%%%%%%%%%%%%%%%%%

\section{Literature Survey}
Ulam stability is introduced via a question by Ulam in monograph \cite{ulam}, which is partially answered by Hyers \cite{hyers} and extended by Rassias \cite{rassias}.
In this manner Ulam stability, also known as Hyers--Ulam stability or Hyers--Ulam--Rassias stability, has developed in the context of differential equations, difference equations (recurrences), functional equations, and operators; see  Brillou\"{e}t--Belluot, Brzd\k{e}k, and Ciepli\'{n}ski \cite{brillouet} for a survey of the literature on this topic, and also {Brzd\k{e}k}, Popa, Ra\c{s}a, and Xu \cite{brzdek1}. Honing in on Ulam stability in the discrete setting, see Popa \cite{popa,popa2}, and more recently Andr\'{a}s and M\'{e}sz\'{a}ros \cite{andras} on time scales, Brzd\k{e}k and W\'{o}jcik \cite{brzdek2}, Hua, Li and Feng \cite{hua}, Jung and Nam \cite{jung}, Nam \cite{nam,nam2,nam3}, Shen \cite{shen,shen2}, and Rasouli, Abbaszadeh, and Eshaghi \cite{rasouli}. 

Fukutaka and Onitsuka \cite{fuon1,fuon2,fuon3} explore on a continuous interval the best constant for Hyers--Ulam stability of both first-order homogeneous linear differential equations with a periodic coefficient, and for a type of Hill's differential equation, thus motivating this exploration of a Hill-type equation for difference equations. See also the related papers by C\v{a}dariu, Popa, and Ra\c{s}a \cite{cad}, and Dragi\v{c}evi\'{c} \cite{drag}. In particular, Fukutaka and Onitsuka \cite{fuon1,fuon2,fuon3} consider the second-order equation $y'' = a(t)y$ and its related Hyers-Ulam stability, where the variable coefficient function takes the form $a(t)=\lambda^2(t)-\lambda'(t)$, where $\lambda$ is a continuously differentiable periodic function of period $\omega>0$. One of their main results is that the Hill equation
\begin{equation}\label{1.1}
 y'' = \Big(\lambda^2(t)-\lambda'(t)\Big)y 
\end{equation}
is Hyers-Ulam stable (HUS) if and only if $\int_0^{\omega}\lambda(t)dt\ne 0$. If $\int_0^{\omega}\lambda(t)dt> 0$, then 
$$  K^*:=\max_{t\in(0,\omega]}\left[e^{-\int_0^t \lambda(s)ds} \int_{-\infty}^{t}\left(e^{2\int_0^s \lambda(u)du}\int_s^{\infty}e^{-\int_0^u \lambda(v)dv}du\right)ds\right] $$
is the minimal (best) Hyers-Ulam constant for \eqref{1.1}. Dragi\v{c}evi\'{c} \cite{drag} followed up on this idea with an associated perturbed equation of the form
\begin{equation}\label{1.2}
 y''=b(t)y'+\Big(\lambda^2(t)-\lambda'(t)+c(t)\Big)y +f(t,y,y'), \quad t\in\R.
\end{equation}
In broad terms, he found that \eqref{1.2} is HUS if $b$ and $c$ are small using the supremum norm; if $f$ satisfies a joint Lipschitz condition in both the second and third variable with small enough Lipschitz constant, and also that $\int_0^{\omega}\lambda(t)dt\ne 0$. As regards a discrete version of this context, there is a gap in the literature, as no one to the best of our knowledge has explored a corresponding Hill-type $h$-difference equation, where we propose below to use the approximation
$$ \lambda^2(t)-\lambda'(t) \approx \lambda(t)\lambda(t+h)-\Delta_h \lambda(t), \quad \Delta_h \lambda(t):=\frac{\lambda(t+h)-\lambda(t)}{h}. $$
In this discrete case, $\lambda$ is periodic in the sense that $\lambda(t+\omega)=\lambda(t)$ for all $t\in\{0,h,2h,3h,\ldots\}$. As we shall see, the techniques employed below will not be a mere discretized recapitulation of the continuous case, but ones developed specifically for, and tailored to, the discrete case with step size $h>0$.

With this stated motivation in mind, we will proceed as follows. In Section 2, we introduce a definition of HUS for two related first-order $h$-difference equations with periodic coefficient, one homogeneous and the other inhomogeneous,  and we present some key constants for periodic $\lambda$ based on its periodic values. A previous important result from the literature on the Hyers-Ulam stability of first-order $h$-difference equations with periodic coefficient, including a Hyers-Ulam stability constant, is modified to fit with these two related equations. In Section 3, a possible Hill-type $h$-difference equation is introduced along the lines outlined above. The main result of this work is presented in this section, namely the conditions under which the proposed discrete Hill equation is Hyers-Ulam stable, and what the Hyers-Ulam stability constant is, for both the second-order homogeneous and inhomogeneous cases. In Section 4, we explore in depth two non-trivial examples, one with a specified period-3 (3-cycle) coefficient, and one with a general period-2 coefficient. Details for both are provided, including HUS constants. In Section 5, we use our robust results from the discrete second-order case to develop wholly new results for certain discrete third-order Hill equations. No similar results for the continuous case exist as far as we know, lending further novelty to this work. A discrete third-order example that utilizes and builds upon an example from Section 4 and the results from Section 5 is provided. A brief concluding remark constitutes Section 6.

%%%%%%%%%%%%%%%%%%%%%%
% Section: 1st-order %
%%%%%%%%%%%%%%%%%%%%%%

\section{First-order periodic coefficient}

Let $h>0$, and define the uniformly discrete set $\T:=\{0,h,2h,3h,\ldots\}$. Define the discrete exponential function via
\begin{equation}\label{epdef}
 e_\lambda(t):=\prod_{k=0}^{\frac{t-h}{h}}\left(1+h\lambda(kh)\right), \quad\text{where}\quad \prod_{k=0}^{-1}f(k)\equiv 1.
\end{equation}
In this section we utilize a coefficient function of arbitrary finite period to study the Hyers-Ulam stability (defined below) of the first-order linear homogeneous difference equation with $n$-cycle (period $n$) coefficient
\begin{equation}\label{neq}
 \Delta_hx(t) - \lambda(t) x(t) = 0, \qquad \Delta_h x(t):=\frac{x(t+h)-x(t)}{h}, 
\end{equation}
where $n\in\N$, $\lambda:\T\rightarrow\R$ is given by
\begin{equation}\label{n-cycle}
 \lambda(t):= \lambda_k \quad\text{if}\quad \frac{t}{h}\equiv k\mod n
\end{equation}
for $k\in\{0,1,\ldots,n-1\}$, and $\lambda_0,\lambda_1,\ldots,\lambda_{n-1}\in\R\backslash\{\frac{-1}{h}\}$ such that the coefficient function $\lambda$ is periodic with period $n$, and $\lambda$ is not periodic for any $k<n$.
It is a routine exercise to check that the discrete exponential function $e_\lambda$ in \eqref{epdef} satisfies \eqref{neq}, with $e_\lambda(0)=1$.
Let $\lambda_0,\lambda_1,\ldots,\lambda_{n-1}\in\C\backslash\{\frac{-1}{h}\}$. 
For convenience, note that
\[ |e_\lambda(kh)|=|1+h\lambda_0| |1+h\lambda_1| \cdots |1+h\lambda_{k-1}|, \quad k\in\{1,2,\ldots,n\}. \]

% Definition 2.1 %

\begin{definition}[Hyers-Ulam Stability]\label{def2.1}
Equation \eqref{neq} has Hyers-Ulam stability if and only if there exists a constant $K>0$ with the following property: For arbitrary $\varepsilon>0$, if a function $\psi:\T\rightarrow\R$ satisfies $|\Delta_h\psi(t)-\lambda(t)\psi(t)|\le\varepsilon$ for all $t\in\T$, then there exists a solution $x:\T\rightarrow\R$ of \eqref{neq} such that $|\psi(t)-x(t)|\le K\varepsilon$ for all $t\in\T$. 
Such a constant $K$ is called a Hyers-Ulam stability constant for \eqref{neq} on $\T$.
\end{definition}

Note that if $|e_\lambda(nh)|=1$, then \eqref{neq} is not Hyers-Ulam stable \cite[Remark 4.1]{aor}. 

% Remark 2.2 %

\begin{remark}
Assume the coefficient function $\lambda$ satisfies \eqref{n-cycle} for $\lambda_0,\lambda_1,\ldots,\lambda_{n-1}\in\R\backslash\{\frac{-1}{h}\}$. Let 
\begin{equation}\label{sum0}
 S_0=S_0(\lambda) =\frac{1}{|1+h\lambda_0|} + \frac{1}{|1+h\lambda_0||1+h\lambda_{1}|} + \cdots + \frac{1}{|1+h\lambda_0||1+h\lambda_1|\cdots|1+h\lambda_{n-1}|} 
\end{equation}
and
\begin{eqnarray}
	S_k &=& S_k(\lambda) \nonumber \\
	&=& \frac{1}{|1+h\lambda_k|} + \frac{1}{|1+h\lambda_k||1+h\lambda_{k+1}|} + \cdots \nonumber \\
	& & + \frac{1}{|1+h\lambda_k|\cdots|1+h\lambda_{n-1}|} +  \frac{1}{|1+h\lambda_k|\cdots|1+h\lambda_{n-1}||1+h\lambda_0|} + \cdots \nonumber \\
	& & + \frac{1}{|1+h\lambda_k|\cdots|1+h\lambda_{n-1}||1+h\lambda_0|\cdots|1+h\lambda_{k-1}|} \label{sumk}
\end{eqnarray}
for $k\in\{1,\ldots,n-1\}$. We will refer to these sums in the following theorem, which is \cite[Theorem 4.4]{aor}. 
\end{remark} 

% Theorem 2.3 %

\begin{theorem}\label{thm2.3}
Assume the coefficient function $\lambda$ satisfies \eqref{n-cycle} for $\lambda_0,\lambda_1,\ldots,\lambda_{n-1}\in\R\backslash\{\frac{-1}{h}\}$, with 
$0 < |e_\lambda(nh)| \ne 1$. Let $S_0=S_0(\lambda)$ and $S_k=S_k(\lambda)$ for $k\in\{1,2,\ldots,n-1\}$ be given by \eqref{sum0} and \eqref{sumk}, respectively. 
Then, \eqref{neq} has Hyers-Ulam stability on $\T$, with Hyers-Ulam stability constant 
\begin{equation}\label{nKmax0}
   K_0(\lambda):=\frac{h|e_\lambda(nh)|}{\left|1-|e_\lambda(nh)|\right|}\max\left\{S_0,S_1,\ldots,S_{n-1}\right\}
\end{equation}
 on $\T$. Moreover, if $|e_\lambda(nh)|>1$, then $K_0$ is the minimum Hyers-Ulam stability constant for \eqref{neq}. 
\end{theorem}

Additionally, consider the related equation
\begin{equation}\label{neq-nonhomo}
 \Delta_h y(t) + \lambda(t) y(t) = f(t), \qquad t\in\T,
\end{equation}
where we have modified \eqref{neq} by replacing $\lambda$ with $-\lambda$, and by adding in the non-homogeneous term $f$.

% Definition 2.4 %

\begin{definition}[Hyers-Ulam Stability]\label{def2.4}
Equation \eqref{neq-nonhomo} has Hyers-Ulam stability if and only if there exists a constant $L>0$ with the following property: For arbitrary $\varepsilon>0$, if a function $\eta:\T\rightarrow\R$ satisfies $\left|\Delta_h\eta(t)+\lambda(t)\eta(t)-f(t)\right|\le\varepsilon$ for all $t\in\T$, then there exists a solution $y:\T\rightarrow\R$ of \eqref{neq-nonhomo} such that $|\eta(t)-y(t)|\le L\varepsilon$ for all $t\in\T$. 
Such a constant $L$ is called a Hyers-Ulam stability constant for \eqref{neq-nonhomo} on $\T$.
\end{definition}

In the following theorem, we derive a new result using Definition \ref{def2.4} that modifies the preceding Theorem \ref{thm2.3} based on Definition \ref{def2.1}, by changing $\lambda$ to $-\lambda$, and by introducing a non-homogeneous term $f$ to the right-hand side of equation \eqref{neq}.

% Theorem 2.5 %

\begin{theorem}\label{thm2.5}
Assume the coefficient function $-\lambda$ satisfies \eqref{n-cycle} for $\lambda_0,\lambda_1,\ldots,\lambda_{n-1}\in\R\backslash\{\frac{1}{h}\}$, with 
$0 < |e_{-\lambda}(nh)| \ne 1$. Let $S_0(-\lambda)$ and $S_k(-\lambda)$ for $k\in\{1,2,\ldots,n-1\}$ be given by \eqref{sum0} and \eqref{sumk}, respectively, with $\lambda$ in those expressions replaced by $-\lambda$. 
Then, \eqref{neq-nonhomo} has Hyers-Ulam stability on $\T$, with Hyers-Ulam stability constant 
\begin{equation}\label{nKmax0-neg}
   K_0(-\lambda):=\frac{h|e_{-\lambda}(nh)|}{\left|1-|e_{-\lambda}(nh)|\right|}\max\left\{S_0(-\lambda),S_1(-\lambda),\ldots,S_{n-1}(-\lambda)\right\}
\end{equation}
 on $\T$. Moreover, if $|e_{-\lambda}(nh)|>1$, then $K_0(-\lambda)$ is the minimum Hyers-Ulam stability constant for \eqref{neq-nonhomo}. 
\end{theorem}

\begin{proof}
Let $u$ be a solution of \eqref{neq-nonhomo} on $\T$, and let $\eta:\T\rightarrow\R$ be a function such that 
$|\Delta_h\eta(t)+\lambda(t)\eta(t)-f(t)|\le\varepsilon$ for all $t\in\T$. Then, using $u$ and $\eta$, we have
\begin{eqnarray*}
 \varepsilon &\ge& \left|\Delta_h\eta(t)+\lambda(t)\eta(t)-f(t)\right| \\
 &=& \left|\Delta_h\eta(t)+\lambda(t)\eta(t)-\left(\Delta_h u(t) + \lambda(t) u(t) \right)\right| \\
 &=& \left|\Delta_h(\eta-u)(t) + \lambda(t)(\eta-u)(t)\right| \\
 &=& \left|\Delta_h(\eta-u)(t) - (-\lambda)(t)(\eta-u)(t)\right|.
\end{eqnarray*}
As in Definition \ref{def2.1}, we have that the function $\psi=(\eta-u)$ satisfies $\left|\Delta_h \psi(t) - (-\lambda)(t)\psi(t)\right|\le\varepsilon$. By Theorem \ref{thm2.3} with $\lambda$ replaced by $-\lambda$, equation \eqref{neq} is Hyers-Ulam stable. Therefore, from Definition \ref{def2.1}, there exists a solution $v:\T\rightarrow\R$ of \eqref{neq}  with $\lambda$ replaced by $-\lambda$, namely $\Delta_h v(t)+\lambda(t)v(t)=0$, such that
\[ \left|\psi(t)-v(t)\right| = \left|(\eta-u)(t)-v(t)\right|\le K_0(-\lambda)\varepsilon. \]
Now let $y(t)=(u+v)(t)$. Then, 
\[ \Delta_h y(t)+\lambda(t) y(t) = \Delta_h (u+v)(t)+\lambda(t)(u+v)(t) = f(t), \]
so that $y$ is a solution of \eqref{neq-nonhomo}, and $|\eta(t)-y(t)|\le K_0(-\lambda)\varepsilon$ for all $t\in\T$. By Definition \ref{def2.4}, equation \eqref{neq-nonhomo} has Hyers-Ulam stability with Hyers-Ulam stability constant $L=K_0(-\lambda)$.
\end{proof}

%%%%%%%%%%%%%
% Section 3 %
%%%%%%%%%%%%%

\section{A possible Hill-type discrete equation}

Using as a foundation the previous section, we introduce completely new results for the remainder of this work.
For given constant step-size $h>0$, define the forward difference operator via
\[ \Delta_h y(t) = \frac{y(t+h)-y(t)}{h} \quad\text{and}\quad \Delta^2_h y(t) = \Delta_h(\Delta_h y(t)), \]
and again let $\T=\{0, h, 2h, \ldots\}$.
Now, consider the discrete Hill-type $h$-difference equation
\begin{equation}\label{hilleq}
 \Delta^2_h y(t) + \left[\Delta_h \lambda(t) - \lambda(t)\lambda(t+h)\right]y(t) = 0,
\end{equation}
where $\lambda:\T\rightarrow\R$ and $\lambda(t+\omega)=\lambda(t)$ for some $\omega\in\T$.

\begin{definition}[Hyers-Ulam Stability]\label{def3.1}
Equation \eqref{hilleq} has Hyers-Ulam stability if and only if there exists a constant $K>0$ with the following property: For arbitrary $\varepsilon>0$, if a function $\xi:\T\rightarrow\R$ satisfies $\left|\Delta^2_h \xi(t) + \left[\Delta_h \lambda(t) - \lambda(t)\lambda(t+h)\right]\xi(t)\right|\le\varepsilon$ for all $t\in\T$, then there exists a solution $y:\T\rightarrow\R$ of \eqref{hilleq} such that $|\xi(t)-y(t)|\le K\varepsilon$ for all $t\in\T$. 
Such a constant $K$ is called a Hyers-Ulam stability constant for \eqref{hilleq} on $\T$.
\end{definition}

Note that if $\lambda(t)\equiv \lambda\in\left(0,\frac{1}{h}\right)\cup\left(\frac{1}{h},\frac{2}{h}\right)\cup\left(\frac{2}{h},\infty\right)$, then \eqref{hilleq} has the form
\[ \Delta^2_h y(t) - \lambda^2y(t) = 0, \]
which is Hyers-Ulam stable \cite[Theorem 3.4]{andon} with best Hyers-Ulam stability constant
\[ K=\begin{cases} \dfrac{1}{\lambda^2} &\text{if}\quad \lambda\in\left(0,\frac{1}{h}\right) \\ \dfrac{h}{\lambda|2-h\lambda|} &\text{if}\quad \lambda\in\left(\frac{1}{h},\frac{2}{h}\right)\cup\left(\frac{2}{h},\infty\right). \end{cases} \]

% Theorem 3.2 %

\begin{theorem}[Main Result]\label{thm3.2}
Assume the coefficient function $\lambda$ satisfies \eqref{n-cycle} for $\lambda_0,\lambda_1,\ldots,\lambda_{n-1}\in\R\backslash\{\pm\frac{1}{h}\}$, with 
$0 < |e_{\pm \lambda}(nh)| \ne 1$, where the discrete exponential function is given in \eqref{epdef}. Let $S_0=S_0(\pm\lambda)$ and $S_k=S_k(\pm\lambda)$ for $k\in\{1,2,\ldots,n-1\}$ be given by \eqref{sum0} and \eqref{sumk}, respectively. Additionally, let $K_0(\pm\lambda)$ be as in \eqref{nKmax0} and \eqref{nKmax0-neg}, respectively. Then, \eqref{hilleq} has Hyers-Ulam stability on $\T$, with Hyers-Ulam stability constant $K=K_0(\lambda)K_0(-\lambda)$.
\end{theorem}

\begin{proof}
For arbitrary $\varepsilon>0$, suppose a function $\xi:\T\rightarrow\R$ satisfies 
$$ \left|\Delta^2_h \xi(t) + \left[\Delta_h \lambda(t) - \lambda(t)\lambda(t+h)\right]\xi(t)\right|\le\varepsilon $$ 
for all $t\in\T$. Define the function $\psi(t):=\Delta_h\xi(t-h)+\lambda(t-h)\xi(t-h)$ for $t\in\T\backslash\{0\}$. Shifting and using the substitution, we have
\begin{align*}
 \Delta_h\psi(t+h)-\lambda(t+h)\psi(t+h) &= \Delta^2_h \xi(t) + \frac{1}{h}\left[\lambda(t+h)\xi(t+h)-\lambda(t)\xi(t)\right] \\
&  -\lambda(t+h)\Delta_h \xi(t)-\lambda(t+h)\lambda(t)\xi(t) \\
&= \Delta^2_h \xi(t) + \frac{1}{h}\left[\lambda(t+h)\xi(t)-\lambda(t)\xi(t)\right]-\lambda(t+h)\lambda(t)\xi(t) \\
&= \Delta^2_h \xi(t) + \left[\Delta_h \lambda(t)-\lambda(t)\lambda(t+h)\right]\xi(t)
\end{align*}
for all $t\in\T$. Shifting the variable back, it follows that $\left|\Delta_h\psi(t)-\lambda(t)\psi(t)\right|\le\varepsilon$ for all $t\in\T\backslash\{0\}$. 
Let $S_0$ and $S_k$ for $k\in\{1,2,\ldots,n-1\}$ be given by \eqref{sum0} and \eqref{sumk}, respectively. 
By Definition \ref{def2.1} and Theorem \ref{thm2.3}, \eqref{neq} has Hyers-Ulam stability with Hyers-Ulam stability constant $K_0(\lambda)$ as in \eqref{nKmax0}. Moreover, if 
$|e_\lambda(nh)|>1$, then $K_0(\lambda)$ is the minimum Hyers-Ulam stability constant for \eqref{neq}. In particular, there exists a solution $x$ of \eqref{neq} such that 
$\left|\psi(t)-x(t)\right|\le K_0(\lambda)\varepsilon$ on $\T\backslash\{0\}$. Using the definition of $\psi$, we have
\[ \left|\Delta_h\xi(t-h)+\lambda(t-h)\xi(t-h)-x(t)\right|=\left|\psi(t)-x(t)\right|\le K_0(\lambda)\varepsilon, \quad t\in\T\backslash\{0\}, \]
so that
\begin{equation*}\label{xi-firstorder}
 \left|\Delta_h\xi(t)+\lambda(t)\xi(t)-x(t+h)\right|\le K_0(\lambda)\varepsilon=\varepsilon', \quad t\in\T. 
\end{equation*}
By Theorem \ref{thm2.5} with $f(t)=x(t+h)$  for all $t\in\T$, there exists a solution $\eta:\T\rightarrow\R$ of the difference equation 
\begin{equation}\label{non-homofirst}
 \Delta_h\eta(t)+\lambda(t)\eta(t)=x(t+h), \quad t\in\T, 
\end{equation}
such that by Definition \ref{def2.4},
\[ \left|\xi(t)-\eta(t)\right|\le K_0(-\lambda)\varepsilon' = K_0(-\lambda)K_0(\lambda)\varepsilon \]
for all $t\in\T$.
Moreover, applying the $h$-difference operator and using the discrete product rule on \eqref{non-homofirst}, we have
\begin{eqnarray*}
 \Delta^2_h\eta(t)+[\Delta_h\lambda(t)-\lambda(t)\lambda(t+h)]\eta(t) &=& \Delta^2_h\eta(t)+\eta(t)\Delta_h\lambda(t)+\lambda(t+h)\Delta_h\eta(t) \\
 & & -\lambda(t+h)\Delta_h\eta(t)-\lambda(t)\lambda(t+h)\eta(t) \\
 &=& \Delta^2_h\eta(t)+\Delta_h\left(\lambda(t)\eta(t)\right) \\
 & & -\lambda(t+h)\left(\Delta_h\eta(t)+\lambda(t)\eta(t)\right) \\
 &=& \Delta_h x(t+h)-\lambda(t+h)x(t+h) \\
 &=& 0
\end{eqnarray*}
for $t\in\T$, as $x=x(t)$ is a solution of \eqref{neq} on $\T\backslash\{0\}$. Therefore, $\eta$ is also a solution of \eqref{hilleq} on $\T$, satisfying
\[ \left|\xi(t)-\eta(t)\right|\le K_0(-\lambda)K_0(\lambda)\varepsilon, \qquad t\in\T. \]
As a consequence, \eqref{hilleq} has Hyers-Ulam stability, with Hyers-Ulam stability constant $K=K_0(\lambda)K_0(-\lambda)$, using \eqref{nKmax0} and \eqref{nKmax0-neg}, respectively.
\end{proof}

Next, we consider the non-homogeneous discrete Hill-type $h$-difference equation
\begin{equation}\label{non-homohilleq}
 \Delta^2_h y(t) + \left[\Delta_h \lambda(t) - \lambda(t)\lambda(t+h)\right]y(t) = f(t),
\end{equation}
where $\lambda$, $f:\T\rightarrow\R$ and $\lambda(t+\omega)=\lambda(t)$ for some $\omega\in\T$.

\begin{definition}[Hyers-Ulam Stability]
Equation \eqref{non-homohilleq} has Hyers-Ulam stability if and only if there exists a constant $L>0$ with the following property: For arbitrary $\varepsilon>0$, if a function $\xi:\T\rightarrow\R$ satisfies $\left|\Delta^2_h \xi(t) + \left[\Delta_h \lambda(t) - \lambda(t)\lambda(t+h)\right]\xi(t)-f(t)\right|\le\varepsilon$ for all $t\in\T$, then there exists a solution $y:\T\rightarrow\R$ of \eqref{non-homohilleq} such that $|\xi(t)-y(t)|\le L\varepsilon$ for all $t\in\T$. 
Such a constant $L$ is called a Hyers-Ulam stability constant for \eqref{non-homohilleq} on $\T$.
\end{definition}

% Theorem 3.4 %

\begin{theorem}\label{thm3.4}
Assume the coefficient function $\lambda$ satisfies \eqref{n-cycle} for $\lambda_0,\lambda_1,\ldots,\lambda_{n-1}\in\R\backslash\{\pm\frac{1}{h}\}$, with 
$0 < |e_{\pm \lambda}(nh)| \ne 1$, where the discrete exponential function is given in \eqref{epdef}. Let $S_0=S_0(\pm\lambda)$ and $S_k=S_k(\pm\lambda)$ for $k\in\{1,2,\ldots,n-1\}$ be given by \eqref{sum0} and \eqref{sumk}, respectively. Additionally, let $K_0(\pm\lambda)$ be as in \eqref{nKmax0} and \eqref{nKmax0-neg}, respectively. Then, \eqref{non-homohilleq} has Hyers-Ulam stability on $\T$, with Hyers-Ulam stability constant $L=K_0(\lambda)K_0(-\lambda)$.
\end{theorem}

\begin{proof}
Let $\varepsilon>0$ be a fixed arbitrary constant, and let $u$ be a solution of \eqref{non-homohilleq} on $\T$. Suppose that $\eta:\T\rightarrow\R$ satisfies
\[ \left|\Delta^2_h \eta(t) + \left[\Delta_h \lambda(t) - \lambda(t)\lambda(t+h)\right]\eta(t)-f(t)\right|\le\varepsilon \]
for all $t\in\T$. Then we have
\begin{eqnarray*}
 \varepsilon &\ge& \left|\Delta^2_h \eta(t) + \left[\Delta_h \lambda(t) - \lambda(t)\lambda(t+h)\right]\eta(t)-f(t)\right| \\
 &=& \left|\Delta^2_h \eta(t) + \left[\Delta_h \lambda(t) - \lambda(t)\lambda(t+h)\right]\eta(t)-\left(\Delta^2_h u(t) + \left[\Delta_h \lambda(t) - \lambda(t)\lambda(t+h)\right]u(t) \right)\right| \\
 &=& \left|\Delta^2_h (\eta-u)(t) + \left[\Delta_h \lambda(t) - \lambda(t)\lambda(t+h)\right](\eta-u)(t)\right|.
\end{eqnarray*}
As in Definition \ref{def3.1}, we have that the function $\psi=(\eta-u)$ satisfies
\[ \left|\Delta^2_h \psi(t) + \left[\Delta_h \lambda(t) - \lambda(t)\lambda(t+h)\right]\psi(t)\right|\le\varepsilon. \]
By Theorem \ref{thm3.2}, there exists a solution $v:\T\rightarrow\R$ of \eqref{hilleq} such that
\[ \left|\psi(t)-v(t)\right| = \left|(\eta-u)(t)-v(t)\right|\le K_0(\lambda)K_0(-\lambda)\varepsilon. \]
Define $y(t)=(u+v)(t)$. Then, we have
\[ \Delta^2_h y(t) + \left[\Delta_h \lambda(t) - \lambda(t)\lambda(t+h)\right]y(t) = \Delta^2_h (u+v)(t) + \left[\Delta_h \lambda(t) - \lambda(t)\lambda(t+h)\right](u+v)(t) = f(t). \]
That is, $y$ is a solution of \eqref{non-homohilleq} such that $|\eta(t)-y(t)|\le K_0(\lambda)K_0(-\lambda)\varepsilon$ for all $t\in\T$. Thus, \eqref{non-homohilleq} has Hyers-Ulam stability with Hyers-Ulam stability constant $L=K_0(\lambda)K_0(-\lambda)$.
\end{proof}

% Lemma 3.3 %

%\begin{lemma}
%If the coefficient function $\lambda$ satisfies \eqref{n-cycle} for $\lambda_0,\lambda_1,\ldots,\lambda_{n-1}\in\R\backslash\{\pm\frac{1}{h}\}$, with 
%$|e_{-\lambda}(nh)| = 1$, where the discrete exponential function is given in \eqref{epdef}, then \eqref{hilleq} is not Ulam stable on $\T$.
%\end{lemma}

%\begin{proof}
%Assume $\lambda$ satisfies \eqref{n-cycle} for $\lambda_0,\lambda_1,\ldots,\lambda_{n-1}\in\R\backslash\{\pm\frac{1}{h}\}$. Let the discrete exponential function be %as in \eqref{epdef}, and suppose $|e_{-\lambda}(nh)| = 1$. Then, 
%$$ \ell:=\max_{t\in\T}|e_{-\lambda}(t+h)| $$
%exists in $\R$ and is finite. Given arbitrary $\varepsilon>0$, define $\xi:\T\rightarrow\R$ via
%\[ \xi(t) = \varepsilon\ell t e_{-\lambda}(t). \]
%Taking $h$-step differences and using the discrete product rule, we have
%\begin{eqnarray*}
% \Delta_h\xi(t)   &=& \varepsilon\ell t (-\lambda(t))e_{-\lambda}(t) + \varepsilon\ell e_{-\lambda}(t+h) \\
% \Delta^2_h\xi(t) &=& \varepsilon\ell t [(-\lambda(t+h))(-\lambda(t))e_{-\lambda}(t)+\Delta_h\lambda(t)e_{-\lambda}(t)] + \varepsilon\ell e_{-\lambda}(t+h) \\
%\end{eqnarray*}
%\end{proof}

%%%%%%%%%%%%%
% Section 4 %
%%%%%%%%%%%%%

\section{Examples}

In this section, we present some examples that apply our results in the previous sections.

% Example 3.3 %

\begin{example}
Fix step size $h>0$. For $A>0$ but $A\not\in\left\{\frac{1}{h}, \frac{\sqrt{2}}{h}\right\}$, set $\lambda(t)=\frac{2A}{\sqrt{3}}\sin\left(\frac{2\pi t}{3h}\right)$. Then, $\lambda$ is a 3-cycle on $\T$, with $\lambda_0=0$, $\lambda_1=A$, and $\lambda_2=-A$. Moreover, $\Delta_h \lambda(t)-\lambda(t)\lambda(t+h)$ is the 3-cycle on $\T$ given by $\left\{\frac{A}{h}, \frac{A(Ah-2)}{h}, \frac{A}{h} \right\}$.
Using \eqref{sum0} and \eqref{sumk}, respectively, we have
\begin{center}
\begin{tabular}{lclllcl}
 $S_0(\lambda)$ & = & $1+\frac{1}{|1+Ah|}+\frac{1}{|1+Ah| |1-Ah|}$, & & $S_0(-\lambda)$ &=& $1+\frac{1}{|1-Ah|}+\frac{1}{|1-Ah| |1+Ah|}$, \\
 $S_1(\lambda)$ & = & $\frac{1}{|1+Ah|} + \frac{2}{|1+Ah| |1-Ah|}$, & & $S_1(-\lambda)$ &=& $\frac{1}{|1-Ah|} + \frac{2}{|1-Ah| |1+Ah|}$, \\
 $S_2(\lambda)$ & = & $\frac{2}{|1-Ah|} + \frac{1}{|1+Ah| |1-Ah|}$, & & $S_2(-\lambda)$ &=& $\frac{2}{|1+Ah|} + \frac{1}{|1-Ah| |1+Ah|}$. 	  
\end{tabular}
\end{center}
By \eqref{epdef} with $t=3h$, we have
\begin{eqnarray*}
 e_{\pm\lambda}(3h) &=& 1 - h^2 A^2, \qquad A\ne\frac{1}{h}.
\end{eqnarray*}
If $|e_\lambda(3h)|=1$, then \eqref{neq} is not Hyers-Ulam stable, and if $|e_{-\lambda}(3h)|=1$, then \eqref{neq-nonhomo} is not Hyers-Ulam stable, by \cite[Remark 4.1]{aor}. 
If $A\ne \frac{\sqrt{2}}{h}$, then $|e_{\pm\lambda}(3h)|\ne 1$. Following calculations on a computer algebra system, we have
\begin{align*} 
\max\left\{S_0(\lambda), S_1(\lambda), S_2(\lambda)\right\} 
 &=\begin{cases} 
	S_2(\lambda) &\text{if}\quad A\in\left(0,\;\frac{1}{h}\right) \cup \left(\frac{1}{h},\;\frac{\sqrt{2}}{h}\right) \cup \left(\frac{\sqrt{2}}{h},\; \frac{1+\sqrt{17}}{2h}\right) \\
	S_0(\lambda) &\text{if}\quad A\in\left[\frac{1+\sqrt{17}}{2h},\; \infty\right),
	\end{cases} \\
\max\left\{S_0(-\lambda), S_1(-\lambda), S_2(-\lambda)\right\} 
 &=\begin{cases} 
	S_1(-\lambda) &\text{if}\quad A\in\left(0,\;\frac{1}{h}\right) \cup \left(\frac{1}{h},\;\frac{\sqrt{2}}{h}\right) \\
	S_0(-\lambda) &\text{if}\quad A\in\left[\frac{\sqrt{2}}{h},\; \infty\right).
	\end{cases} 
\end{align*}
Employing Theorem \ref{thm3.2} for step size $h>0$ but $h\not\in\left\{\frac{1}{A},\; \frac{\sqrt{2}}{A}\right\}$, equation \eqref{hilleq} has Hyers-Ulam stability on $\T$, with Hyers-Ulam stability constant $K=K_0(\lambda)K_0(-\lambda)$, where from \eqref{nKmax0} and \eqref{nKmax0-neg}
\begin{eqnarray*}
   K_0(\lambda):=\frac{h|e_{\lambda}(3h)|}{\left|1-|e_{\lambda}(3h)|\right|}\max\left\{S_0(\lambda), S_1(\lambda), S_2(\lambda)\right\}, \\
   K_0(-\lambda):=\frac{h|e_{-\lambda}(3h)|}{\left|1-|e_{-\lambda}(3h)|\right|}\max\left\{S_0(-\lambda), S_1(-\lambda), S_2(-\lambda)\right\}.
\end{eqnarray*}
To be precise, for any step size $h>0$ and any $A>0$ but $A\not\in\left\{\frac{1}{h},\; \frac{\sqrt{2}}{h}\right\}$, equation \eqref{hilleq} has Hyers-Ulam stability on 
$\T$, with Hyers-Ulam stability constant $K=K_0(\lambda)K_0(-\lambda)$ given by
\[ K= \frac{h^2\left|1-h^2A^2\right|^2}{\left|1-|1-h^2A^2|\right|^2}
\begin{cases}
S_2(\lambda)S_1(-\lambda) &\text{if}\quad A\in\left(0,\;\frac{1}{h}\right) \cup \left(\frac{1}{h},\;\frac{\sqrt{2}}{h}\right) \\
S_2(\lambda)S_0(-\lambda) &\text{if}\quad A\in\left(\frac{\sqrt{2}}{h},\; \frac{1+\sqrt{17}}{2h}\right) \\
S_0(\lambda)S_0(-\lambda) &\text{if}\quad A\in\left[\frac{\sqrt{2}}{h},\; \infty\right). 
\end{cases} \]
\end{example}

% Example 3.4: Two-Cycle Case %

\begin{example}\label{Two-Cycle}
Fix step size $h>0$. For $A,B>0$ with $A\ne B$ and $A,B\not\in\left\{\frac{1}{h}\right\}$, set 
$$ \lambda(t) = \begin{cases} A &\text{if}\quad \frac{t}{h}\equiv 0 \mod 2, \\ B &\text{if}\quad \frac{t}{h}\equiv 1 \mod 2.
\end{cases} $$ 
Then, $\lambda$ is a 2-cycle on $\T$, with $\lambda_0=A$ and $\lambda_1=B$. Moreover, $\Delta_h \lambda(t)-\lambda(t)\lambda(t+h)$ is the 2-cycle on $\T$ given by $\left\{\frac{B-A(1+Bh)}{h}, \frac{A-B(1+Ah)}{h} \right\}$.
Using \eqref{sum0} and \eqref{sumk}, respectively, we have
\begin{center}
\begin{tabular}{lclllcl}
 $S_0(\lambda)$ & = & $\frac{1}{|1+Ah|}+\frac{1}{|1+Ah| |1+Bh|}$, & & $S_0(-\lambda)$ &=& $\frac{1}{|1-Ah|}+\frac{1}{|1-Ah| |1-Bh|}$, \\
 $S_1(\lambda)$ & = & $\frac{1}{|1+Bh|}+\frac{1}{|1+Ah| |1+Bh|}$, & & $S_1(-\lambda)$ &=& $\frac{1}{|1-Bh|}+\frac{1}{|1-Ah| |1-Bh|}$.
\end{tabular}
\end{center}
By \eqref{epdef} with $t=2h$, we have
\begin{eqnarray*}
 e_{\pm\lambda}(2h) &=& (1\pm Ah)(1\pm Bh), \qquad A,B\ne\frac{1}{h}.
\end{eqnarray*}
If $|e_{\lambda}(2h)|=1$, then \eqref{neq} is not Hyers-Ulam stable, and if $|e_{-\lambda}(2h)|=1$, then \eqref{neq-nonhomo} is not Hyers-Ulam stable, by \cite[Remark 4.1]{aor}. 
Working through the cases, we see that if $A\ne \frac{B}{-1+Bh}$ $\left(h\ne\frac{A+B}{AB}\right)$ for $B>\frac{1}{h}$, and $A\ne \frac{2-Bh}{h(1-Bh)}$ for $B\in\left(0,\frac{1}{h}\right)\cup\left(\frac{2}{h},\infty\right)$, then $|e_{\pm\lambda}(2h)|\ne 1$. 
Calculations using a computer algebra system show that
\begin{align*} 
\max\left\{S_0(\lambda), S_1(\lambda)\right\} 
 &=\begin{cases} 
	S_0(\lambda) &\text{if}\quad 0<A\le B \\
	S_1(\lambda) &\text{if}\quad A>B,
	\end{cases} \\
\max\left\{S_0(-\lambda), S_1(-\lambda)\right\} 
 &=\begin{cases}
	S_0(-\lambda) &\text{if}\quad 
	\begin{cases} 
	  0<A<B &\text{and}\quad h\in\left(\frac{2}{A+B},\frac{1}{A}\right)\cup\left(\frac{1}{A},\infty\right),\; \text{or} \\ A>B>0 &\text{and}\quad h\in\left(0,\frac{1}{A}\right)   
		       \cup\left(\frac{1}{A},\frac{2}{A+B}\right)
	\end{cases} \\
	S_1(-\lambda) &\text{if}\quad 
	\begin{cases} 
	  0<A<B &\text{and}\quad h\in\left(0,\frac{1}{B}\right)\cup\left(\frac{1}{B},\frac{2}{A+B}\right),\; \text{or} \\ A>B>0 &\text{and}\quad 
		h\in\left(\frac{2}{A+B},\frac{1}{B}\right)\cup\left(\frac{1}{B},\infty\right).
	\end{cases} \\\end{cases} 
\end{align*}
Employing Theorem \ref{thm3.2} for step size $h>0$ but $h\not\in\left\{\frac{1}{A},\; \frac{1}{B}\right\}$, equation \eqref{hilleq} has Hyers-Ulam stability on $\T$, with Hyers-Ulam stability constant $K=K_0(\lambda)K_0(-\lambda)$, where from \eqref{nKmax0} and \eqref{nKmax0-neg}
$$ K_0(\lambda):=\frac{h|e_{\lambda}(2h)|}{\left|1-|e_{\lambda}(2h)|\right|}\max\left\{S_0(\lambda), S_1(\lambda)\right\}, \quad K_0(-\lambda):=\frac{h|e_{-\lambda}(2h)|}{\left|1-|e_{-\lambda}(2h)|\right|}\max\left\{S_0(-\lambda), S_1(-\lambda)\right\}. $$
To be precise, for any step size $h>0$ and any $A>0$ but $A,B\not\in\left\{\frac{1}{h}\right\}$, equation \eqref{hilleq} has Hyers-Ulam stability on 
$\T$, with Hyers-Ulam stability constant $K=K_0(\lambda)K_0(-\lambda)$ given by
\begin{align*}
& K = \frac{h^2 \left|1-A^2h^2\right| \left|1-B^2h^2\right|}{\left|1-|1+Ah||1+Bh|\right| \left|1-|1-Ah||1-Bh|\right|} \\
&\cdot\begin{cases}
S_0(\lambda)S_0(-\lambda) &\text{if}\quad 0<A<B \;\text{and}\; h\in\left(\frac{2}{A+B},\frac{1}{A}\right)\cup\left(\frac{1}{A},\frac{A+B}{AB}\right)\cup\left(\frac{A+B}{AB},\infty\right) \\
S_0(\lambda)S_1(-\lambda) &\text{if}\quad 0<A<B \;\text{and}\; h\in\left(0,\frac{1}{B}\right)\cup\left(\frac{1}{B},\frac{2}{A+B}\right) \\
S_1(\lambda)S_0(-\lambda) &\text{if}\quad A>B>0 \;\text{and}\; h\in\left(0,\frac{1}{A}\right)\cup\left(\frac{1}{A},\frac{2}{A+B}\right) \\
S_1(\lambda)S_1(-\lambda) &\text{if}\quad A>B>0 \;\text{and}\; h\in\left(\frac{2}{A+B},\frac{1}{B}\right)\cup\left(\frac{1}{B},\frac{A+B}{AB}\right)\cup\left(\frac{A+B}{AB},\infty\right).
\end{cases} 
\end{align*}
This completes the example.
\end{example}

%%%%%%%%%%%%%
% Section 5 %
%%%%%%%%%%%%%

\section{Application to third-order equations}

In this section, we establish the Hyers-Ulam stability for certain third-order linear $h$-difference equations by applying the main theorem. Now, we consider the third-order $h$-difference equation
\begin{equation}\label{third}
 \Delta^3_h y(t) + p(t)\Delta^2_h y(t) + q(t)\Delta_h y(t) + r(t)y(t) = 0
\end{equation}
with
\begin{equation}\label{pqr}
 \begin{split}
  p(t) &= \lambda(t+2h), \\
  q(t) &= 2\Delta_h \lambda(t+h) +\Delta_h \lambda(t+2h) - \lambda(t+2h)\lambda(t+3h), \\
  r(t) &= \Delta^2_h \lambda(t) + \lambda(t)\Delta_h \lambda(t+2h) - \lambda(t)\lambda(t+2h)\lambda(t+3h),
 \end{split}
\end{equation}
where $\lambda:\T\rightarrow\R$ and $\lambda(t+\omega)=\lambda(t)$ for some $\omega\in\T$.
\begin{definition}[Hyers-Ulam Stability]
Equation \eqref{third} has Hyers-Ulam stability if and only if there exists a constant $K>0$ with the following property: For arbitrary $\varepsilon>0$, if a function $\xi:\T\rightarrow\R$ satisfies $\left|\Delta^3_h \xi(t) + p(t)\Delta^2_h \xi(t) + q(t)\Delta_h \xi(t) + r(t)\xi(t)\right|\le\varepsilon$ for all $t\in\T$, then there exists a solution $y:\T\rightarrow\R$ of \eqref{third} such that $|\xi(t)-y(t)|\le K\varepsilon$ for all $t\in\T$. 
Such a constant $K$ is called a Hyers-Ulam stability constant for \eqref{third} on $\T$.
\end{definition}

% Theorem 5.2 %

\begin{theorem}\label{thm5.2}
Assume the coefficient function $\lambda$ satisfies \eqref{n-cycle} for $\lambda_0,\lambda_1,\ldots,\lambda_{n-1}\in\R\backslash\{\pm\frac{1}{h}\}$, with 
$0 < |e_{\pm \lambda}(nh)| \ne 1$, where the discrete exponential function is given in \eqref{epdef}. Let $S_0=S_0(\pm\lambda)$ and $S_k=S_k(\pm\lambda)$ for $k\in\{1,2,\ldots,n-1\}$ be given by \eqref{sum0} and \eqref{sumk}, respectively. Additionally, let $K_0(\pm\lambda)$ be as in \eqref{nKmax0} and \eqref{nKmax0-neg}, respectively. Then, \eqref{third} with \eqref{pqr} has Hyers-Ulam stability on $\T$, with Hyers-Ulam stability constant $K=K_0(\lambda)(K_0(-\lambda))^2$.
\end{theorem}

\begin{proof}
Let $\varepsilon>0$ be a fixed arbitrary constant. Suppose that $\xi:\T\rightarrow\R$ satisfies
$$ \left|\Delta^3_h \xi(t) + p(t)\Delta^2_h \xi(t) + q(t)\Delta_h \xi(t) + r(t)\xi(t)\right|\le\varepsilon $$ 
for all $t\in\T$, where $p$, $q$ and $r$ are given in \eqref{pqr}. Let $\psi(t):=\Delta_h\xi(t-2h)+\lambda(t-2h)\xi(t-2h)$ for $t\in\T\backslash\{0,h\}$. Since
\[ \Delta_h\psi(t+2h) = \Delta^2_h \xi(t) + \lambda(t+h)\Delta_h\xi(t) + \left(\Delta_h\lambda(t)\right)\xi(t) \]
and
\[ \Delta^2_h\psi(t+2h) = \Delta^3_h \xi(t) + \lambda(t+2h)\Delta^2_h \xi(t) + 2\left(\Delta_h \lambda(t+h) \right)\Delta_h\xi(t) + \left(\Delta^2_h \lambda(t) \right) \xi(t) \]
hold, we see that
\begin{align*}
 \Delta^2_h &\psi(t+2h) + \left[\Delta_h \lambda(t+2h)-\lambda(t+2h)\lambda(t+3h)\right]\psi(t+2h)\\
  &= \Delta^3_h \xi(t) + \lambda(t+2h)\Delta^2_h \xi(t) + 2\left(\Delta_h \lambda(t+h) \right)\Delta_h\xi(t) + \left(\Delta^2_h \lambda(t) \right) \xi(t) \\
  &\hspace{5mm} + \left[\Delta_h \lambda(t+2h)-\lambda(t+2h)\lambda(t+3h)\right]\left[\Delta_h\xi(t)+\lambda(t)\xi(t)\right] \\
  &= \Delta^3_h \xi(t) + p(t)\Delta^2_h \xi(t) + q(t)\Delta_h \xi(t) + r(t)\xi(t)
\end{align*}
for all $t\in\T$. Shifting the variable back, it follows that
\[ \left|\Delta^2_h\psi(t)+[\Delta_h\lambda(t)-\lambda(t)\lambda(t+h)]\psi(t)\right|\le\varepsilon \]
for all $t\in\T\backslash\{0,h\}$. 
Let $S_0=S_0(\pm\lambda)$ and $S_k=S_k(\pm\lambda)$ for $k\in\{1,2,\ldots,n-1\}$ be given by \eqref{sum0} and \eqref{sumk}, respectively. Additionally, let $K_0(\pm\lambda)$ be as in \eqref{nKmax0} and \eqref{nKmax0-neg}, respectively. 
Using Theorem \ref{thm3.2}, equation \eqref{hilleq} has Hyers-Ulam stability with Hyers-Ulam stability constant $K_0(\lambda)K_0(-\lambda)$. That is, there is a solution $x$ of \eqref{hilleq} such that 
$\left|\psi(t)-x(t)\right|\le K_0(\lambda)K_0(-\lambda)\varepsilon$ on $\T\backslash\{0,h\}$. From the definition of $\psi$, we obtain the inequality
\[ \left|\Delta_h\xi(t-2h)+\lambda(t-2h)\xi(t-2h)-x(t)\right|=\left|\psi(t)-x(t)\right|\le K_0(\lambda)K_0(-\lambda)\varepsilon, \quad t\in\T\backslash\{0,h\}, \]
so that
\[ \left|\Delta_h\xi(t)+\lambda(t)\xi(t)-x(t+2h)\right|\le K_0(\lambda)K_0(-\lambda)\varepsilon=\tilde{\varepsilon}, \quad t\in\T. \]
Using Theorem \ref{thm2.5} with $f(t)=x(t+2h)$ for all $t\in\T$, there exists a solution $\eta:\T\rightarrow\R$ of the difference equation 
\begin{equation}\label{non-homofirst2}
 \Delta_h\eta(t)+\lambda(t)\eta(t)=x(t+2h), \quad t\in\T, 
\end{equation}
such that
\[ \left|\xi(t)-\eta(t)\right|\le K_0(-\lambda)\tilde{\varepsilon} = K_0(\lambda)(K_0(-\lambda))^2\varepsilon \]
for all $t\in\T$.

Next, we will prove that $\eta$ is also a solution of \eqref{third} with \eqref{pqr} on $\T$. 
From \eqref{non-homofirst2} and the discrete product rule, we have
\[ \Delta^2_hx(t+2h) = \Delta^3_h \eta(t) + \lambda(t+2h)\Delta^2_h \eta(t) + 2\left(\Delta_h \lambda(t+h) \right)\Delta_h\eta(t) + \left(\Delta^2_h \lambda(t) \right) \eta(t) \]
for all $t\in\T$. Hence we obtain
\begin{align*}
 \Delta^3_h \eta(t) &+ p(t)\Delta^2_h \eta(t) + q(t)\Delta_h \eta(t) + r(t)\eta(t)\\
  &= \Delta^3_h \eta(t) + \lambda(t+2h)\Delta^2_h \eta(t) + 2\left(\Delta_h \lambda(t+h) \right)\Delta_h\eta(t) + \left(\Delta^2_h \lambda(t) \right) \eta(t) \\
  &\hspace{5mm} + \left[\Delta_h \lambda(t+2h)-\lambda(t+2h)\lambda(t+3h)\right]\left[\Delta_h\eta(t)+\lambda(t)\eta(t)\right] \\
  &= \Delta^2_hx(t+2h)+\left[\Delta_h \lambda(t+2h)-\lambda(t+2h)\lambda(t+3h)\right]x(t+2h)\\
  &=0
\end{align*}
for $t\in\T$. Therefore, $\eta$ is also a solution of \eqref{third} with \eqref{pqr} on $\T$, and so that, \eqref{third} with \eqref{pqr} has Hyers-Ulam stability, with Hyers-Ulam stability constant $K=K_0(\lambda)(K_0(-\lambda))^2$, using \eqref{nKmax0} and \eqref{nKmax0-neg}, respectively.
\end{proof}

Next, we consider \eqref{third} with
\begin{equation}\label{pqr2}
 \begin{split}
  p(t) &= -\lambda(t+2h), \\
  q(t) &= -2\Delta_h \lambda(t+h) +\Delta_h \lambda(t+2h) - \lambda(t+2h)\lambda(t+3h), \\
  r(t) &= -\Delta^2_h \lambda(t) - \lambda(t)\Delta_h \lambda(t+2h) + \lambda(t)\lambda(t+2h)\lambda(t+3h),
 \end{split}
\end{equation}
where $\lambda:\T\rightarrow\R$ and $\lambda(t+\omega)=\lambda(t)$ for some $\omega\in\T$.

% Theorem 5.3 %

\begin{theorem}\label{thm5.3}
Assume the coefficient function $\lambda$ satisfies \eqref{n-cycle} for $\lambda_0,\lambda_1,\ldots,\lambda_{n-1}\in\R\backslash\{\pm\frac{1}{h}\}$, with 
$0 < |e_{\pm \lambda}(nh)| \ne 1$, where the discrete exponential function is given in \eqref{epdef}. Let $S_0=S_0(\pm\lambda)$ and $S_k=S_k(\pm\lambda)$ for $k\in\{1,2,\ldots,n-1\}$ be given by \eqref{sum0} and \eqref{sumk}, respectively. Additionally, let $K_0(\pm\lambda)$ be as in \eqref{nKmax0} and \eqref{nKmax0-neg}, respectively. Then, \eqref{third} with \eqref{pqr2} has Hyers-Ulam stability on $\T$, with Hyers-Ulam stability constant $K=(K_0(\lambda))^2K_0(-\lambda)$.
\end{theorem}

\begin{proof}
Let $\varepsilon>0$ be a fixed arbitrary constant. Suppose that $\xi:\T\rightarrow\R$ satisfies
$$ \left|\Delta^3_h \xi(t) + p(t)\Delta^2_h \xi(t) + q(t)\Delta_h \xi(t) + r(t)\xi(t)\right|\le\varepsilon $$ 
for all $t\in\T$, where $p$, $q$ and $r$ are given by \eqref{pqr2}. Define $\psi(t):=\Delta_h\xi(t-2h)-\lambda(t-2h)\xi(t-2h)$ for $t\in\T\backslash\{0,h\}$. Since
\[ \Delta_h\psi(t+2h) = \Delta^2_h \xi(t) - \lambda(t+h)\Delta_h\xi(t) - \left(\Delta_h\lambda(t)\right)\xi(t) \]
and
\[ \Delta^2_h\psi(t+2h) = \Delta^3_h \xi(t) - \lambda(t+2h)\Delta^2_h \xi(t) - 2\left(\Delta_h \lambda(t+h) \right)\Delta_h\xi(t) - \left(\Delta^2_h \lambda(t) \right) \xi(t) \]
hold, we see that
\begin{align*}
 \Delta^2_h &\psi(t+2h) + \left[\Delta_h \lambda(t+2h)-\lambda(t+2h)\lambda(t+3h)\right]\psi(t+2h)\\
  &= \Delta^3_h \xi(t) - \lambda(t+2h)\Delta^2_h \xi(t) - 2\left(\Delta_h \lambda(t+h) \right)\Delta_h\xi(t) - \left(\Delta^2_h \lambda(t) \right) \xi(t) \\
  &\hspace{5mm} + \left[\Delta_h \lambda(t+2h)-\lambda(t+2h)\lambda(t+3h)\right]\left[\Delta_h\xi(t)-\lambda(t)\xi(t)\right] \\
  &= \Delta^3_h \xi(t) + p(t)\Delta^2_h \xi(t) + q(t)\Delta_h \xi(t) + r(t)\xi(t)
\end{align*}
for all $t\in\T$. Hence we have
\[ \left|\Delta^2_h\psi(t)+[\Delta_h\lambda(t)-\lambda(t)\lambda(t+h)]\psi(t)\right|\le\varepsilon \]
for all $t\in\T\backslash\{0,h\}$. By Theorem \ref{thm3.2}, we conclude that there exists a solution $x$ of \eqref{hilleq} such that
$\left|\psi(t)-x(t)\right|\le K_0(\lambda)K_0(-\lambda)\varepsilon$ on $\T\backslash\{0,h\}$. This implies that
\[ \left|\Delta_h\xi(t)-\lambda(t)\xi(t)-x(t+2h)\right|\le K_0(\lambda)K_0(-\lambda)\varepsilon=\tilde{\varepsilon}, \quad t\in\T. \]
Using Theorem \ref{thm2.5} with $f(t)=x(t+2h)$ and $\lambda$ replaced by $-\lambda$ for all $t\in\T$, there exists a solution $\eta:\T\rightarrow\R$ of the difference equation 
\[ \Delta_h\eta(t)-\lambda(t)\eta(t)=x(t+2h), \quad t\in\T \]
such that
\[ \left|\xi(t)-\eta(t)\right|\le K_0(\lambda)\tilde{\varepsilon} = (K_0(\lambda))^2K_0(-\lambda)\varepsilon \]
for all $t\in\T$. We can show that $\eta$ is a solution of \eqref{third} with \eqref{pqr2} on $\T$ using the same technique of Theorem \ref{thm5.2}. Therefore, \eqref{third} with \eqref{pqr2} has Hyers-Ulam stability, with Hyers-Ulam stability constant $K=(K_0(\lambda))^2K_0(-\lambda)$.
\end{proof}

Next, we consider \eqref{third} with
\begin{equation}\label{pqr3}
 \begin{split}
  p(t) &= \lambda(t+2h), \\
  q(t) &= \Delta_h \lambda(t+h) - \lambda(t+h)\lambda(t+2h), \\
  r(t) &= \Delta^2_h \lambda(t) - \lambda(t)\Delta_h \lambda(t+h) - \lambda(t)\lambda(t+h)\lambda(t+2h),
 \end{split}
\end{equation}
where $\lambda:\T\rightarrow\R$ and $\lambda(t+\omega)=\lambda(t)$ for some $\omega\in\T$.

% Theorem 5.4 %

\begin{theorem}\label{thm5.4}
Assume the coefficient function $\lambda$ satisfies \eqref{n-cycle} for $\lambda_0,\lambda_1,\ldots,\lambda_{n-1}\in\R\backslash\{\pm\frac{1}{h}\}$, with 
$0 < |e_{\pm \lambda}(nh)| \ne 1$, where the discrete exponential function is given in \eqref{epdef}. Let $S_0=S_0(\pm\lambda)$ and $S_k=S_k(\pm\lambda)$ for $k\in\{1,2,\ldots,n-1\}$ be given by \eqref{sum0} and \eqref{sumk}, respectively. Additionally, let $K_0(\pm\lambda)$ be as in \eqref{nKmax0} and \eqref{nKmax0-neg}, respectively. Then, \eqref{third} with \eqref{pqr3} has Hyers-Ulam stability on $\T$, with Hyers-Ulam stability constant $K=K_0(\lambda)(K_0(-\lambda))^2$.
\end{theorem}

\begin{proof}
Let $\varepsilon>0$ be a fixed arbitrary constant. Suppose that $\xi:\T\rightarrow\R$ satisfies
$$ \left|\Delta^3_h \xi(t) + p(t)\Delta^2_h \xi(t) + q(t)\Delta_h \xi(t) + r(t)\xi(t)\right|\le\varepsilon $$ 
for all $t\in\T$, where $p$, $q$ and $r$ are given by \eqref{pqr3}. Define
\[ \psi(t):=\Delta^2_h\xi(t-2h)+[\Delta_h\lambda(t-2h)-\lambda(t-2h)\lambda(t-h)]\xi(t-2h). \]
Since
\begin{align*}
 \Delta_h\psi(t+2h) &= \Delta^3_h \xi(t) + \left[\Delta_h \lambda(t+h)-\lambda(t+h)\lambda(t+2h)\right] \Delta_h \xi(t) \\
 &\hspace{5mm} + \left[\Delta^2_h\lambda(t) - \lambda(t+h)\left(\Delta_h\lambda(t+h)+\Delta_h\lambda(t)\right)\right]\xi(t)
\end{align*}
holds, we see that
\begin{align*}
 \Delta_h &\psi(t+2h) + \lambda(t+2h)\psi(t+2h)\\
  &= \Delta^3_h \xi(t) + \lambda(t+2h)\Delta^2_h \xi(t) + \left[\Delta_h \lambda(t+h)-\lambda(t+h)\lambda(t+2h)\right] \Delta_h \xi(t) \\
  &\hspace{5mm} + \left\{\Delta^2_h\lambda(t) - \lambda(t+h)\left[\Delta_h\lambda(t+h)+\Delta_h\lambda(t)\right] + \lambda(t+2h)\left[\Delta_h \lambda(t)-\lambda(t)\lambda(t+h)\right] \right\} \xi(t) \\
  &= \Delta^3_h \xi(t) + p(t)\Delta^2_h \xi(t) + q(t)\Delta_h \xi(t) + r(t)\xi(t)
\end{align*}
for all $t\in\T$. Hence we have
\[ \left|\Delta_h\psi(t)+\lambda(t)\psi(t)\right|\le\varepsilon \]
for all $t\in\T\backslash\{0,h\}$. By Theorem \ref{thm2.5} with $f(t)\equiv 0$, we conclude that there exists a solution $x$ of \eqref{neq-nonhomo} with $f(t)\equiv 0$ such that
$\left|\psi(t)-x(t)\right|\le K_0(-\lambda)\varepsilon$ on $\T\backslash\{0,h\}$. This implies that
\[ \left|\Delta^2_h\xi(t)+[\Delta_h\lambda(t)-\lambda(t)\lambda(t+h)]\xi(t)-x(t+2h)\right|\le K_0(-\lambda)\varepsilon=\tilde{\varepsilon}, \quad t\in\T. \]
Using Theorem \ref{thm3.4} with $f(t)=x(t+2h)$, there exists a solution $\eta:\T\rightarrow\R$ of the difference equation 
\begin{equation}\label{non-homo2nd}
 \Delta^2_h\eta(t)+[\Delta_h\lambda(t)-\lambda(t)\lambda(t+h)]\eta(t)=x(t+2h), \quad t\in\T
\end{equation}
such that
\[ \left|\xi(t)-\eta(t)\right|\le K_0(\lambda)K_0(-\lambda)\tilde{\varepsilon} = K_0(\lambda)(K_0(-\lambda))^2\varepsilon \]
for all $t\in\T$. 

Next, we will prove that $\eta$ is also a solution of \eqref{third} with \eqref{pqr3} on $\T$. 
From \eqref{non-homo2nd} and the discrete product rule, we have
\begin{align*}
 \Delta_hx(t+2h) &= \Delta^3_h \eta(t) + \left[\Delta_h \lambda(t+h)-\lambda(t+h)\lambda(t+2h)\right] \Delta_h \eta(t) \\
 &\hspace{5mm} + \left[\Delta^2_h\lambda(t) - \lambda(t+h)\left(\Delta_h\lambda(t+h)+\Delta_h\lambda(t)\right)\right]\eta(t)
\end{align*}
for all $t\in\T$. Hence we obtain
\begin{align*}
 \Delta^3_h \eta(t) &+ p(t)\Delta^2_h \eta(t) + q(t)\Delta_h \eta(t) + r(t)\eta(t)\\
  &= \Delta^3_h \eta(t) + \lambda(t+2h)\Delta^2_h \eta(t) + \left[\Delta_h \lambda(t+h)-\lambda(t+h)\lambda(t+2h)\right] \Delta_h \eta(t) \\
  &\hspace{5mm} + \left\{\Delta^2_h\lambda(t) - \lambda(t+h)\left[\Delta_h\lambda(t+h)+\Delta_h\lambda(t)\right] + \lambda(t+2h)\left[\Delta_h \lambda(t)-\lambda(t)\lambda(t+h)\right] \right\} \eta(t) \\
  &= \Delta_h x(t+2h) + \lambda(t+2h)x(t+2h)\\
  &= 0
\end{align*}
for $t\in\T$. Therefore, $\eta$ is also a solution of \eqref{third} with \eqref{pqr3} on $\T$, and so that, \eqref{third} with \eqref{pqr3} has Hyers-Ulam stability, with Hyers-Ulam stability constant $K=K_0(\lambda)(K_0(-\lambda))^2$.
\end{proof}

Finally, we consider \eqref{third} with coefficient functions
\begin{equation}\label{pqr4}
 \begin{split}
  p(t) &= -\lambda(t+2h), \\
  q(t) &= \Delta_h \lambda(t+h) - \lambda(t+h)\lambda(t+2h), \\
  r(t) &= \Delta^2_h \lambda(t) - \lambda(t+h)\Delta_h \lambda(t+h) - \left(\lambda(t+h)+\lambda(t+2h)\right)\Delta_h \lambda(t)\\
       &\hspace{5mm} + \lambda(t)\lambda(t+h)\lambda(t+2h),
 \end{split}
\end{equation}
where $\lambda:\T\rightarrow\R$ and $\lambda(t+\omega)=\lambda(t)$ for some $\omega\in\T$.

% Theorem 5.5 %

\begin{theorem}\label{thm5.5}
Assume the coefficient function $\lambda$ satisfies \eqref{n-cycle} for $\lambda_0,\lambda_1,\ldots,\lambda_{n-1}\in\R\backslash\{\pm\frac{1}{h}\}$, with 
$0 < |e_{\pm \lambda}(nh)| \ne 1$, where the discrete exponential function is given in \eqref{epdef}. Let $S_0=S_0(\pm\lambda)$ and $S_k=S_k(\pm\lambda)$ for $k\in\{1,2,\ldots,n-1\}$ be given by \eqref{sum0} and \eqref{sumk}, respectively. Additionally, let $K_0(\pm\lambda)$ be as in \eqref{nKmax0} and \eqref{nKmax0-neg}, respectively. Then, \eqref{third} with \eqref{pqr4} has Hyers-Ulam stability on $\T$, with Hyers-Ulam stability constant $K=(K_0(\lambda))^2K_0(-\lambda)$.
\end{theorem}

\begin{proof}
Theorem \ref{thm5.5} is proved using the same method as Theorem \ref{thm5.4}. Note that we now have to create an expression of the form $\Delta_h \psi(t+2h) - \lambda(t+2h)\psi(t+2h)$, where
\[ \psi(t):=\Delta^2_h\xi(t-2h)+[\Delta_h\lambda(t-2h)-\lambda(t-2h)\lambda(t-h)]\xi(t-2h). \]
This will complete the proof in an analogous manner.
\end{proof}

% Example 5.6: Two-Cycle %

\begin{example}
Fix step size $h>0$ with $h \ne \frac{1}{\pi}$, $\frac{1}{2\pi}$, $\frac{2}{3\pi}$. Now we consider \eqref{third} with \eqref{pqr} on $\T$. Set
$$ \lambda(t) = \begin{cases} \pi &\text{if}\quad \frac{t}{h}\equiv 0 \mod 2, \\ 2\pi &\text{if}\quad \frac{t}{h}\equiv 1 \mod 2.
\end{cases} $$ 
Then, $\lambda$ is a 2-cycle on $\T$, with $\lambda_0=\pi$ and $\lambda_1=2\pi$, and
\[ \begin{split}
  p(t) &= \lambda(t+2h) = \lambda(t), \\
  q(t) &= 2\Delta_h \lambda(t+h) +\Delta_h \lambda(t+2h) - \lambda(t+2h)\lambda(t+3h) \\
       &= 2\Delta_h \lambda(t+h) +\Delta_h \lambda(t) - \lambda(t)\lambda(t+h) \\
       &= \begin{cases} -\frac{\pi}{h}-2\pi^2 &\text{if}\quad \frac{t}{h}\equiv 0 \mod 2, \\ \frac{\pi}{h}-2\pi^2 &\text{if}\quad \frac{t}{h}\equiv 1 \mod 2,
\end{cases} \\
  r(t) &= \Delta^2_h \lambda(t) + \lambda(t)\Delta_h \lambda(t+2h) - \lambda(t)\lambda(t+2h)\lambda(t+3h) \\
       &= \Delta^2_h \lambda(t) + \lambda(t)\Delta_h \lambda(t) - \lambda^2(t)\lambda(t+h) \\
       &= \begin{cases} -\frac{2\pi}{h^2}+\frac{\pi^2}{h}-2\pi^3 &\text{if}\quad \frac{t}{h}\equiv 0 \mod 2, \\ \frac{4\pi}{h^2}-\frac{2\pi^2}{h}-4\pi^3 &\text{if}\quad \frac{t}{h}\equiv 1 \mod 2.
\end{cases} \\
 \end{split} \]
From Example \ref{Two-Cycle} with $A=\pi$ and $B=2\pi$, we have
\begin{center}
\begin{tabular}{lclllcl}
 $S_0(\lambda)$ & = & $\frac{1}{1+h\pi}+\frac{1}{(1+h\pi)(1+2h\pi)}$, & & $S_0(-\lambda)$ &=& $\frac{1}{|1-h\pi|}+\frac{1}{|1-h\pi| |1-2h\pi|}$, \\
 $S_1(\lambda)$ & = & $\frac{1}{1+2h\pi}+\frac{1}{(1+h\pi)(1+2h\pi)}$, & & $S_1(-\lambda)$ &=& $\frac{1}{|1-2h\pi|}+\frac{1}{|1-h\pi| |1-2h\pi|}$,
\end{tabular}
\end{center}
and
\begin{eqnarray*}
 e_{\pm\lambda}(2h) &=& (1\pm h\pi)(1\pm 2h\pi).
\end{eqnarray*}
Moreover, we see that
\begin{align*} 
\max\left\{S_0(\lambda), S_1(\lambda)\right\} 
 &=S_0(\lambda), \\
\max\left\{S_0(-\lambda), S_1(-\lambda)\right\} 
 &=\begin{cases}
	S_1(-\lambda) &\text{if}\quad h\in\left(0,\frac{1}{2\pi}\right)\cup\left(\frac{1}{2\pi},\frac{2}{3\pi}\right), \\
	S_0(-\lambda) &\text{if}\quad h\in\left(\frac{2}{3\pi},\frac{1}{\pi}\right)\cup\left(\frac{1}{\pi},\infty\right). \\\end{cases} 
\end{align*}
Employing Theorem \ref{thm5.2} for step size $h>0$ with $h \ne \frac{1}{\pi}$, $\frac{1}{2\pi}$, $\frac{2}{3\pi}$, equation \eqref{third} with \eqref{pqr} has Hyers-Ulam stability on $\T$, with Hyers-Ulam stability constant $K=K_0(\lambda)(K_0(-\lambda))^2$, where from \eqref{nKmax0} and \eqref{nKmax0-neg}
$$ K_0(\lambda):=\frac{h|e_{\lambda}(2h)|}{\left|1-|e_{\lambda}(2h)|\right|}\max\left\{S_0(\lambda), S_1(\lambda)\right\} = \frac{2h(1+h\pi)}{(1+h\pi)(1+2h\pi)-1}  = \frac{2(1+h\pi)}{\pi(3+2h\pi)} , $$
and
\begin{align*}
 K_0(-\lambda) &:=\frac{h|e_{-\lambda}(2h)|}{\left|1-|e_{-\lambda}(2h)|\right|}\max\left\{S_0(-\lambda), S_1(-\lambda)\right\} = \begin{cases}
	\frac{h(2-h\pi)}{\left|1-|e_{-\lambda}(2h)|\right|} &\text{if}\quad h\in\left(0,\frac{1}{2\pi}\right)\cup\left(\frac{1}{2\pi},\frac{2}{3\pi}\right), \\
	\frac{2h^2\pi}{\left|1-|e_{-\lambda}(2h)|\right|} &\text{if}\quad h\in\left(\frac{2}{3\pi},\frac{1}{\pi}\right)\cup\left(\frac{1}{\pi},\infty\right)
  \end{cases}\\
  &= \begin{cases}
	\frac{2-h\pi}{\pi(3-2h\pi)} &\text{if}\quad h\in\left(0,\frac{1}{2\pi}\right), \\
        \frac{h(2-h\pi)}{2h^2\pi^2-3h\pi+2} &\text{if}\quad h\in\left(\frac{1}{2\pi},\frac{1}{\pi}\right), \\
        \frac{2-h\pi}{\pi(3-2h\pi)} &\text{if}\quad h\in\left(\frac{1}{\pi},\frac{3}{2\pi}\right), \\
	\frac{2h}{2h\pi-3} &\text{if}\quad h\in\left(\frac{2}{3\pi},\infty\right). \\\end{cases}
\end{align*}
To be precise, Hyers-Ulam stability constant $K=K_0(\lambda)(K_0(-\lambda))^2$ is given by
\[ K = \begin{cases}
	\frac{2(1+h\pi)(2-h\pi)^2}{\pi^3(3+2h\pi)(3-2h\pi)^2} &\text{if}\quad h\in\left(0,\frac{1}{2\pi}\right), \\
        \frac{2h^2(1+h\pi)(2-h\pi)^2}{\pi(3+2h\pi)\left(2h^2\pi^2-3h\pi+2\right)^2} &\text{if}\quad h\in\left(\frac{1}{2\pi},\frac{1}{\pi}\right), \\
        \frac{2(1+h\pi)(2-h\pi)^2}{\pi^3(3+2h\pi)(3-2h\pi)^2} &\text{if}\quad h\in\left(\frac{1}{\pi},\frac{3}{2\pi}\right), \\
	\frac{8h^2(1+h\pi)}{(3+2h\pi)(3-2h\pi)^2} &\text{if}\quad h\in\left(\frac{2}{3\pi},\infty\right). \\\end{cases}
\]
This ends the final example.
\end{example}

\section{Conclusion}
In this work, we utilized previous results on first-order linear $h$-difference equations with a periodic coefficient to establish the Hyers-Ulam stability, and find a Hyers-Ulam constant, for a linear second-order Hill-type $h$-difference equation with a periodic coefficient of a particular form. Specifically, if the discrete exponential function is not of modulus one when evaluated at $nh$, where $n\in\N$ is the period and $h>0$ is the step size, then the equation is Hyers-Ulam stable. We also provide detailed information on the Hyers-Ulam constant. Building on this, the Hyers-Ulam stability of several related third-order $h$-difference equations is established. Examples are provided for both the second- and third-order cases. 

\section*{Acknowledgements}
The second author was supported by JSPS KAKENHI Grant Number JP20K03668.

%%%%%%%%%%%%%%%% 
% Bibliography % 
%%%%%%%%%%%%%%%%

\end{document}